\documentclass[11pt,oneside,reqno, titlepage]{amsart}
\usepackage[utf8]{inputenc}
\usepackage[T1]{fontenc}
\usepackage{lmodern}
\usepackage{amssymb, amsmath, amsthm, mathtools, cases,enumerate}
\usepackage{color}
\usepackage{hyperref,mathscinet}
\usepackage{natbib}
\usepackage{microtype}
\usepackage{hyphenat}
\usepackage{dsfont}

\usepackage{amsaddr}

\newcommand{\TITLE}{Quantitative analysis of optimal Sobolev-Lorentz embeddings with \texorpdfstring{$\alpha$}{alpha}-homogeneous weights}
\newcommand{\AUTHORS}{Petr Gurka, Jan Lang, Zden\v{e}k Mihula}
\newcommand{\ABSTRAKT}{Optimal weighted Sobolev-Lorentz embeddings with homogeneous weights in open convex cones are established, with the exact value of the optimal constant. These embeddings are non-compact, and this paper investigates the structure of their non-compactness quantitatively. Opposite to the previous results in this direction, the non-compactness in this case does not occur uniformly over all sub-domains of the underlying domain, and the problem is not translation invariant, and so these properties cannot be exploited here. Nevertheless, by developing a new approach based on a delicate interplay between the size of suitable extremal functions and the size of their supports, the exact values of the (ball) measure of non-compactness and of all injective strict \texorpdfstring{$s$}{s}-numbers (in particular, of the Bernstein numbers) are obtained. Moreover, it is also shown that the embedding is not strictly singular.}

\hypersetup{
	unicode=true,
	breaklinks=true,
	pdftitle={\TITLE},
	pdfauthor={\AUTHORS},
	pdfsubject={\ABSTRAKT}
	}

\title{\TITLE}
\author{\AUTHORS}

\address[A1]{Petr Gurka, Department of Mathematics, Czech University of Life Sciences Prague, 165 21, Prague 6, Czech Republic; Department of Mathematics, College of Polytechnics Jihlava, Tolstého 16, 586 01, Jihlava, Czech Republic; Department of Mathematical Analysis, Faculty of Mathematics and Physics, Charles University, Sokolovská 83, 186 75 Praha 8, Czech Republic\\ \href{https://orcid.org/0000-0002-0995-4711}{\itshape ORCID~0000-0002-0995-4711}}
\email{gurka@tf.czu.cz}

\address{Jan Lang, Department of Mathematics, The Ohio State University, 231 West 18th Avenue, Columbus, OH 43210-1174, USA; Czech Technical University in Prague, Faculty of Electrical Engineering, Department of Mathematics, Technick\'a~2, 166~27 Praha~6, Czech Republic\\ \href{https://orcid.org/0000-0003-1582-7273}{\itshape ORCID~0000-0003-1582-7273}}
\email{lang.162@osu.edu}

\address{Zden\v ek Mihula, Czech Technical University in Prague, Faculty of Electrical Engineering, Department of Mathematics, Technick\'a~2, 166~27 Praha~6, Czech Republic\\ \href{https://orcid.org/0000-0001-6962-7635}{\itshape ORCID~0000-0001-6962-7635}}
\email[Corresponding author]{mihulzde@fel.cvut.cz}

\numberwithin{equation}{section}

\thanks{This research was partly supported by grant No.~23-04720S of the Czech Science Foundation, and by project OPVVV CAAS CZ.02.1.01/0.0/0.0/16\_019/0000778.}

\DeclareMathOperator{\spn}{span}

\theoremstyle{plain}
\newtheorem{thm}{Theorem}[section]

\newtheorem{prop}[thm]{Proposition}

\theoremstyle{definition}

\newtheorem{rem}[thm]{Remark}
\theoremstyle{remark}

\def\R{{\mathbb R}}
\def\Rde{{\R^d}}
\def\Ne{{\mathbb{N}}}

\def\al{\alpha}

\def\dd{\text{\rm\,d}} 
\def\spt{\operatorname{supp}}  

\begin{document}

\setcitestyle{numbers}
\bibliographystyle{abbrvnat}

\begin{abstract}
\ABSTRAKT
\end{abstract}

\maketitle

\section{Introduction}
It is a truth generally acknowledged that Sobolev embeddings hold a prominent position in various areas of mathematics, making comprehensive understanding of their internal structure and behavior essential. One of their oft-studied aspects is their compactness and its quality. The quality of compactness is quite often analyzed through the (ball) measure of non-compactness or the decay rate of different $s$-numbers. The former is a quantity of a geometric nature (introduced by Kuratowski in~\cite{Kur:30}), whereas the latter is an axiomatic approach (introduced by Pietsch in~\cite{P:74}) suitable for studying fine properties of operators (see~Section~\ref{sec:prel} for precise definitions). Loosely speaking, they can be used to obtain quantitative information about ``how much compact'' an operator is, or about, when the operator is not compact, whether it has some better properties than those possessed by merely bounded operators.

The quality of compactness of compact Sobolev embeddings has been intensively studied, and it is closely related to the spectral theory of the corresponding differential operators associated with them and provides estimates for the growth of their eigenvalues (e.g., see~\cite{CS:90, EE:18, ET:96}). Although surprisingly less attention has been devoted to studying the structure of non-compact Sobolev embeddings, their structure is indisputably of interest, too. For example, their measure of non-compactness may be related to the shape of the essential spectrum (see~\cite{EE:18}), and the rate of decay of their $s$-numbers can be used for deriving that they belong to ``better classes'' of operators than that of merely bounded ones, such as the class of (finitely) strictly singular operators (see~\cite{BG:89, LaMi23}, cf.~\cite{LR-P:14}). This paper studies certain (weighted) non-compact Sobolev embeddings (see below for more information), shedding some light on the structure of their non-compactness.

There are several ways that can cause non-compactness of Sobolev embeddings, such as:
\begin{enumerate}[(a)]
	\item when the underlying domain is unbounded (see~\cite{AFbook}, cf.~\cite{EML:21});
	\item when the boundary of the underlying domain is excessively irregular (see~\cite{KMR:01, LM:08, Mabook, MP:97};
	\item when the target function norm is too strong (see~\cite{KP:08, LMP:22, S:15} and references therein).
\end{enumerate}
Among these possibilities, the last one is particularly intriguing because it has not been explored quantitatively much, despite the long-standing interest in optimal Sobolev embeddings (i.e., those equipped with, in a sense, the strongest possible target function norm; e.g., see \cite{CPS:15} and references therein). Previous works investigating the case (c) (see \cite{B:20, H:03, LaMi23, LMOP:21}) dealt with standard (unweighted) Sobolev embeddings on bounded domains. Their important feature exploited there is that they are translation invariant, and
, loosely speaking, their non-compactness is spread uniformly over all sub-domains.

In this paper, we will consider quite a general non-compact weighted Sobolev embedding, whose non-compactness does not occur uniformly over all sub-domains of the underlying domain; instead, the embedding may become non-compact only in the sub-domains touching the boundary of the underlying domain but remains compact in all other bounded sub-domains (with regular boundary) of the domain. The Sobolev embedding in question is
\begin{equation}\label{intro:emb}
E\colon V^1 L^{p,q}(\Sigma, \mu) \to L^{p^*, q}(\Sigma, \mu).
\end{equation}
Here $\Sigma\subseteq \Rde$, $d\geq2$, is an open convex cone with vertex at the origin, endowed with a nonnegative weight $w$ that is $\alpha$-homogeneous, $\alpha>0$, and its $(1/\alpha)$-th power is concave. The Sobolev space $V^1 L^{p,q}(\Sigma, \mu)$ is a suitable weighted Sobolev space built on the Lorentz space $L^{p,q}(\Sigma, \mu)$, $1\leq q\leq p < D$, $D=d+\al$, $\mu$ is the weighted measure $\dd\mu(x)=w(x) \dd x$, and $p^* = (Dp)/(D - p)$ (see Section~\ref{sec:prel} for precise definitions). An important example of such weights is the monomial weight defined as $w(x)=x_1^{A_1} \cdots x_k^{A_k}$ on $\Sigma = \{x_i>0, i = 1, \dots, k\}$, where $1\leq k \leq d$ and $A_i > 0$. These monomial weights have been quite fashionable since they appeared in \cite{CR-O:13, CR-O:13b} (cf.~\cite{CGPR-O:22}), in connection with the regularity of stable solutions to certain planar reaction--diffusion problems (see~\cite{CR-OS:16} for other interesting examples). We will obtain the exact value of the measure of non-compactness and of the so-called Bernstein numbers of the embedding, which apart from their connection with the spectral and operator theory are useful for proving lower bounds on various approximation quantities (e.g., see~\cite{DNS:06, K:16, P:85}).

A few comments are in order. First, the Sobolev embedding \eqref{intro:emb} generalizes and improves \cite[Inequality (1.3)]{CR-OS:16}, which is restricted to Lebesgue spaces. The Sobolev space there is built on the Lebesgue space $L^p(\Sigma, \mu)$, which coincides with the Lorentz space $L^{p,q}(\Sigma, \mu)$ when $q = p$, and the target function space is the Lebesgue space $L^{p^*}(\Sigma, \mu)$, which is strictly bigger (i.e., its function norm is essentially weaker) than the Lorentz space $L^{p^*, p}(\Sigma, \mu)$ (see~\cite[Chapter~4, Section~4]{BS}). When this paper was submitted, the weighted Sobolev-Lorentz embedding \eqref{intro:emb} itself was new in this generality, to the best of our knowledge. Later, the validity of Sobolev embeddings on convex cones endowed with $\alpha$-homogeneous weights was thoroughly studied in the framework of rearrangement-invariant function spaces in \cite{D:24}. The Sobolev embedding \eqref{intro:emb} is contained in \cite[Theorem~4.1]{D:24}, which also shows that the Lorentz space $L^{p^*, p}(\Sigma, \mu)$ on the right-hand side of \eqref{intro:emb} cannot be replaced by an essentially smaller rearrangement-invariant function space. Nevertheless, our proof of \eqref{intro:emb} is more direct and gives immediately the value of the embedding constant. Secondly, even though the open convex cone $\Sigma$ is unbounded, the lack of compactness is not caused by its unboundedness. The embedding is still non-compact when the underlying domain is replaced by $\Sigma \cap B_R$, where $B_R$ is the open ball centered at the origin with radius $R>0$. Next, the regularity of $\partial\Sigma$ is also completely immaterial. Lastly, while the weights considered in this paper may or may not be singular near parts of $\partial\Sigma$, they are always singular at the origin.

The paper is structured as follows. In the next section, we introduce the necessary notation and recall some preliminary results. In Section~\ref{sec:emb}, we prove the Pólya-Szeg\H{o} inequality for \eqref{intro:emb} (Theorem~\ref{thm:PS}), and then prove the optimal weighted Sobolev-Lorentz inequality in $\Sigma$ (Theorem~\ref{thm:Sob-Lor_emb}). Moreover, we obtain the inequality with the optimal constant (Remark~\ref{rem:optimal_constant}). Finally, in Section~\ref{sec:bern}, we will prove that the embedding is so-called maximally noncompact, i.e., its (ball) measure of non-compactness is equal to its norm. We also obtain the exact values of the Bernstein numbers of the embedding \eqref{intro:emb} (Theorem~\ref{thm:bernstein_numbers}), showing that they all coincide with its norm. Moreover, leveraging the fact that the non-compactness is the worst in neighborhoods of the origin, we construct an infinitely-dimensional subspace of $V^1 L^{p,q}(\Sigma, \mu)$ restricted onto which the embedding is an isomorphism into $L^{p\, ,^{\!\!\!*}q}(\Sigma, \mu)$ (Proposition \ref{prop:almost_extremal_system}), which shows that the embedding is not strictly singular.


\section{Preliminaries}\label{sec:prel}

This section contains basic definitions and necessary preliminary results.

\subsection*{Open convex cone and weighted measure}
Throughout the paper $\Sigma$ denotes an open convex cone in $\Rde$, $d\in\Ne$, $d\geq2$, with vertex at the origin.
Furthermore, $w\colon \overline{\Sigma} \to [0, \infty)$ is a nonnegative (not identically zero) continuous function that is $\al$-homogeneous, $\al>0$,
and such that $w^{1/\alpha}$ is concave in $\Sigma$. Recall that $w$ is $\al$-homogeneous if
\begin{equation} \label{Wfunc}
  w(\kappa x)=\kappa^\al w(x)\quad\text{for any $x\in\overline{\Sigma}$ and all $\kappa>0$}.
\end{equation}
Throughout the paper $\mu$ denotes the weighted measure on $\Sigma$ defined as
\begin{equation*}
  \dd\mu(x)=w(x) \dd x.
\end{equation*}
For future reference, we set
\begin{equation}\label{NumD}
  D=d+\al.
\end{equation}

A large number of interesting examples of cones $\Sigma$ and weights $w$ satisfying the assumptions can be found in \cite[Section~2]{CR-OS:16}. For example, the following weights on $\Sigma = \Sigma_1 \cap \cdots \cap \Sigma_k$, $k\in\{1, \dots, d\}$, where $\Sigma_j\neq\Rde$, $j=1, \dots, k$, are open convex cones in $\Rde$ with vertex at the origin. Given $A_1, \dots, A_k>0$, the weight $w\colon \overline{\Sigma} \to [0, \infty)$ defined as
\begin{equation*}
w(x) = \prod_{j = 1}^k \operatorname{dist}(x, \partial \Sigma_j)^{A_j},\ x\in \overline{\Sigma},
\end{equation*}
satisfies the assumptions. In particular, when $\Sigma_j = \{x\in\Rde: x_j > 0\}$, $w$ is the monomial weight defined as (see~\cite{CR-O:13})
\begin{equation*}
w(x) = x_1^{A_1}\cdots x_k^{A_k},\ x\in \overline{\Sigma}.
\end{equation*}

\subsection*{Lebesgue space}
We denote by $L^p(\Sigma,\mu)$, $p\in[1,\infty]$, the \emph{Lebesgue space} of all  measurable functions~$f$ in $\Sigma$ with finite norm
\begin{equation*}
  \|f\|_{p,\mu}=
  \bigg\{
    \begin{array}{ll} \vspace{3pt}
      \big( \int_\Sigma |f(x)|^p\dd \mu(x) \big)^{1/p} & \hbox{if $p\in[1, \infty)$,} \\
      \text{{$\mu$}-ess sup\,}_{x\in\Sigma}|f(x)| & \hbox{if $p=\infty$.}
    \end{array}
\end{equation*}

\subsection*{Distribution function and rearrangements}
For a measurable function  $f$ in $\Sigma$ we define its \emph{distribution function with respect to $\mu$} as
\begin{equation*}
  f_{*\mu}(\tau)=\mu\big(\{x\in \Sigma: |f(x)|>\tau\}\big), \quad \tau>0,
\end{equation*}
and its \emph{nonincreasing rearrangement with respect to $\mu$} as
\begin{equation*}
  f^*_\mu(t)=\inf\{\tau>0: f_{*\mu}(\tau)\le t\},\quad t>0.
\end{equation*}
The function $f^\bigstar_\mu$ defined as
\begin{equation*}
f^\bigstar_\mu(x) = f^*_\mu(C_D |x|^D),\quad x\in\Rde,
\end{equation*}
where
\begin{equation}\label{unit_ball_in_sigma_measure}
C_D = \mu(B_1\cap \Sigma),
\end{equation}
is the \emph{radially nonincreasing rearrangement} of $f$ \emph{with respect to} $\mu$. We use the notation
\begin{equation*}
  B_r=\{x\in\Rde: |x|<r\},\quad r>0.
\end{equation*}
Thanks to the $\al$-homogeneity of $w$, we have
\begin{equation*}
\mu(B_r\cap \Sigma) = C_D r^D.
\end{equation*}
Note that the function $f^\bigstar_\mu$ is nonnegative and radially nonincreasing. Though it is defined on the whole $\Rde$, it depends only on function values of $f$ in $\Sigma$. Furthermore, the functions $f$ and $f^\bigstar_\mu$ are equimeasurable with respect to $\mu$, that is,
\begin{equation*}
  \mu\big(\{x\in \Sigma: |f(x)|>\tau\}\big)
  =
  \mu\big(\{x\in\Sigma: |f^\bigstar_\mu(x)|>\tau\}\big), \quad \tau>0.
\end{equation*}

When $E\subseteq(0, \infty)$ is (Lebesgue) measurable, we denote its Lebesgue measure by $|E|$. For a measurable function $\phi$ in $(0, \infty)$, we define its \emph{nonincreasing rearrangement} (with respect to the Lebesgue measure) as
\begin{equation*}
\phi^*(t) = \inf \{\tau > 0: |\{s\in(0, \infty): |\phi(s)| > \tau\}| \leq t\},\ t\in(0, \infty).
\end{equation*}

\subsection*{Lorentz space}
Let $p,q\in[1, \infty]$. Assume that either $1\leq q\leq p < \infty$ or $p = q = \infty$. We denote by $L^{p,q}(\Sigma,\mu)$ the set of all measurable functions $f$ in $\Sigma$ such that
\begin{equation*}
\|f\|_{p,q,\mu} = \|t^{\frac1{p} - \frac1{q}} f^*_\mu(t)\|_{L^q(0, \infty)} < \infty.
\end{equation*}
Under the imposed restriction on the parameters $p$ and $q$, the functional $\|\cdot\|_{p,q,\mu}$ is a norm on $L^{p,q}(\Sigma,\mu)$ (e.g., \cite[Chapter~4, Theorem~4.3]{BS}). The function space $L^{p,q}(\Sigma,\mu)$ is usually called the \emph{Lorentz space} with indices $p, q$. Thanks to the layer cake representation formula (e.g., \cite[Chapter~1, Proposition~1.8]{BS}), we have $\|\cdot\|_{p,p,\mu} = \|\cdot\|_{p,\mu}$. The Lorentz norm $\|\cdot\|_{p,q,\mu}$ can be expressed in terms of the distributional function as
\begin{equation}\label{prel:Lorentz_norm_distributional_function}
\|f\|_{p,q,\mu} = \Big( p \int_0^\infty t^{q - 1} f_{*\mu}(t)^\frac{q}{p} \dd t\Big)^{\frac1{q}}
\end{equation}
for every measurable functions $f$ in $\Sigma$. The Lorentz spaces are nested in the sense that
\begin{equation*}
L^{p,q_1}(\Sigma,\mu) \subsetneq L^{p,q_2}(\Sigma,\mu) \quad \text{when $1\leq q_1 < q_2 \leq p < \infty$}.
\end{equation*}

Note that the functions equimeasurable with respect to $\mu$ have the same $\|\cdot\|_{p,q,\mu}$ norm. In particular,
\begin{equation*}
\|f\|_{p,q,\mu} = \|f^\bigstar_\mu\|_{p,q,\mu}
\end{equation*}
for every measurable function $f$ in $\Sigma$.

When $1\leq q\leq p < \infty$, the Lorentz norm $\|\cdot\|_{p,q,\mu}$ is \emph{absolutely continuous}, that is,
\begin{equation*}
  \|f\chi_{E_n}\|_{p,q,\mu}\to0 \quad\text{for every sequence
  $\{\chi_{E_n}\}_{n=1}^\infty$ satisfying $E_n\to\emptyset$ $\mu$-a.e.}
\end{equation*}
Here $E_n\to\emptyset$ $\mu$-a.e. means that the sequence $\{\chi_{E_n}\}_{n=1}^\infty$ converges pointwise to $0$ $\mu$-a.e.

\subsection*{Sobolev-Lorentz space}
Let $1\leq q\leq p < \infty$. We denote by $V^1L^{p,q}(\Sigma, \mu)$ the completion of $\mathcal C_c^1(\Rde)$\textemdash the space of continuously differentiable functions with compact
support in $\Rde$\textemdash with respect to the norm
\begin{equation*}
\|u\|_{V^1L^{p,q}(\Sigma, \mu)} = \|\nabla u\|_{p,q,\mu};
\end{equation*}
two functions that coincide $\mu$-a.e.\ in $\Sigma$ are identified. Note that the functions from $V^1L^{p,q}(\Sigma, \mu)$ need not vanish on $\partial\Sigma$.
We write $\|\nabla u\|_{p,q,\mu}$ for short instead of $\||\nabla u|\|_{p,q,\mu}$, where $|\nabla u|$ stands for the Euclidean norm of $\nabla u$.

\subsection*{Measure of non-compactness, Bernstein numbers, and strictly singular operators}
Throughout this subsection we assume that $X$ and $Y$ are Banach spaces. We denote by $B(X,Y)$ the collection of all bounded linear operators from  $X$ to  $Y$, and by $B_X$ the closed unit ball of $X$ centered at the origin.

The \emph{(ball) measure of non-compactness} $\beta(T)$ of  $T\in B(X,Y)$ is defined as
\begin{equation*}
\beta(T) = \inf\{r>0: \text{$T(B_X)$ can be covered by finitely many balls in $Y$ with radius $r$}\}.
\end{equation*}
Clearly $0 \leq \beta(T) \leq \|T\|$, and it is easy to see that $T$ is compact if and only if $\beta(T) = 0$. The measure of non-compactness is a geometric quantity that, from the point of view of finite coverings, measures how far from the class of compact operators the operator $T$ is (e.g., see~\cite{BG:80, EE:18} for more information). We say that the operator $T$ is \emph{maximally non-compact} if $\beta(T) = \|T\|$.

The $n$th, $n\in\Ne$, \emph{Bernstein number} $b_n(T)$ of $T\in B(X,Y)$ is defined as
\begin{equation*}
b_n(T)=\sup_{X_n \subseteq X} \inf_{\substack{x\in X_n\\ \|x\|_{X} = 1}}\|Tx\|_{Y},
\end{equation*}
where the supremum extends over all $n$-dimensional subspaces of $X$.

Bernstein numbers are an important example of so-called \emph{(strict) $s$-numbers}. Any rule $s\colon T\to\left\{s_n(T)\right\}_{n=1}^\infty$ that assigns each bounded linear operator $T$ from $X$ to $Y$ a sequence $\left\{s_n(T)\right\}_{n=1}^\infty$ of nonnegative numbers having, for every $n\in\Ne$, the following properties:
\begin{itemize}
\item[(S1)] $\|T\|=s_1(T)\geq s_2(T)\geq\cdots\geq0$;
\item[(S2)] $s_n(S+T)\leq s_n(S)+\|T\|$ for every $S\in B(X,Y)$;
\item[(S3)] $s_n(BTA)\leq\|B\|s_n(T)\|A\|$ for every $A\in B(W,X)$ and $B\in B(Y,Z)$, where $W,Z$ are Banach spaces;
\item[(S4)] $s_n(\operatorname{Id}_E\colon E \to E)=1$ for every Banach space $E$ with $\dim E\geq n$;
\item[(S5)] $s_n(T)=0$ if $\operatorname{rank} T< n$
\end{itemize}
is called a (strict) $s$-number. Notable examples of (strict) $s$-numbers are the approximation numbers $a_n$, the Bernstein numbers $b_n$, the Gelfand numbers $c_n$, the Kolmogorov numbers $d_n$, the isomorphism numbers $i_n$, or the Mityagin numbers $m_n$. For their definitions and the difference between strict $s$-numbers and `non-strict' $s$-numbers, the interested reader is referred to \citep[Chapter~5]{EL:11} and references therein.

An operator $T\in B(X,Y)$ is said to be \emph{strictly singular} if there is no infinite-dimensional closed subspace $Z$ of $X$ such that the restriction $T\rvert_{Z}$ of $T$ to $Z$ is an isomorphism of ${Z}$ onto $T(Z)\subseteq Y$. Equivalently, for each infinite-dimensional (closed) subspace $Z$ of $X$,
\begin{equation*}
\inf\left\{  \left\Vert Tx\right\Vert _{Y}\colon \left\Vert x\right\Vert _{X}=1,x\in
Z\right\}  = 0.
\end{equation*}


\section{Optimal Sobolev-Lorentz embedding}\label{sec:emb}
We start by proving a suitable P\'olya-Szeg\H{o} inequality (see~\cite[Chapter~3]{B:19} and references therein).
\begin{thm}\label{thm:PS}
Let $1\leq q\leq p < \infty$. For every function $u\in\mathcal C_c^1(\Rde)$, its radially nonincreasing rearrangement $u^\bigstar_\mu$ with respect to $\mu$ belongs to $V^1L^{p,q}(\Sigma, \mu)$, and
\begin{equation}\label{thm:PS:ineq}
\|\nabla u^\bigstar_\mu\|_{p,q,\mu} \leq \|\nabla u\|_{p,q,\mu}.
\end{equation}
Furthermore, the function $u^*_\mu$ is locally absolutely continuous in $(0, \infty)$, and
\begin{equation}\label{thm:PS:u_bigstar_as_integral}
u^\bigstar_\mu(x) = \int_{C_D|x|^D}^\infty (-u^*_\mu)'(t) \dd t \quad \text{for every $x\in \Rde$},
\end{equation}
where $D$ is defined by \eqref{NumD} and $C_D$ by \eqref{unit_ball_in_sigma_measure}. In particular,
$\spt u^\bigstar_\mu = \overline{B}_R$ where $R$ is such that $\mu(B_R \cap \Sigma) = \mu(\spt u \cap \Sigma)$.
\end{thm}
\begin{proof}
The proof of \eqref{thm:PS:ineq} follows the line of argument of Talenti (\cite{T:76, T:97}),
which combines the coarea formula with a suitable isoperimetric inequality
(see also the proof of \cite[Lemma~4.1]{CP:98}). The suitable isoperimetric inequality
in our case is that from \cite[Theorem~1.3]{CR-OS:16} (cf.~\cite{CR-O:13, CGPR-O:22}).
It reads as, for every measurable $E\subseteq\Rde$ such that $\mu(E\cap \Sigma) < \infty$,
\begin{equation*}
P_w(E; \Sigma)\geq D C_D^\frac1{D} \mu(E\cap \Sigma)^\frac{D - 1}{D}.
\end{equation*}
Here $P_w(E; \Sigma)$ is the weighted perimeter of $E$ in $\Sigma$ (its definition in full generality
can be found in \cite[p.~2977]{CR-OS:16}). Recall that, when $P_w(E; \Sigma) < \infty$, then
\begin{equation*}
P_w(E; \Sigma) = \int_{\partial^*E\cap\Sigma}w(x)\dd\mathcal H^{n-1}(x),
\end{equation*}
where $\partial^*E$ is the reduced boundary of $E$. Note that
\begin{equation*}
D C_D^\frac1{D} = \frac{P_w(B_1; \Sigma)}{\mu(B_1\cap\Sigma)^\frac{D - 1}{D}}
\end{equation*}
thanks to the homogeneity of $w$. Following Talenti, it can be shown that the function $u^*_\mu$
is locally absolutely continuous in $(0, \infty)$ and that, for a.e.~$t\in(0, \infty)$,
\begin{equation}\label{thm:PS:eq1}
0\leq ( D C_D^\frac1{D} t^{\frac{D - 1}{D}} (-u_\mu^*)'(t) )^q \leq \frac{\mathrm d}{\mathrm dt}
\int_{\{x\in\Sigma: |u(x)| > u^*_\mu(t)\}} |\nabla u(x)|^q \dd \mu(x).
\end{equation}
The validity of \eqref{thm:PS:u_bigstar_as_integral} is an immediate consequence of the local absolute
continuity of $u_\mu^*$ combined with the fact that $\lim_{t\to \infty} u_\mu^*(t) = 0$.

Let
\begin{equation}\label{thm:PS:def_phi}
\phi(t) = D C_D^\frac1{D} t^{\frac{D - 1}{D}} (-u_\mu^*)'(t),\ t\in(0, \infty).
\end{equation}
We claim that
\begin{equation}\label{thm:PS:eq2}
\int_0^t \phi^*(s)^q \dd s\leq \int_0^t (\nabla u)_\mu^*(s)^q \dd s \quad \text{for every $t\in(0, \infty)$}.
\end{equation}
Thanks to \cite[Chapter~2, Proposition~3.3]{BS}, it is sufficient to prove that
\begin{equation*}
\int_E \phi(s)^q \dd s \leq \int_0^t (\nabla u)_\mu^*(s)^q \dd s \quad \text{for every measurable $E\subseteq(0, \infty)$, $|E| = t$}.
\end{equation*}
Moreover, it follows from the (outer) regularity of the Lebesgue measure that we may assume that $E$ is open. Thus
\begin{equation*}
E = \bigcup_{j\in\mathcal J}(a_j, b_j),
\end{equation*}
where $\{(a_j, b_j)\}_{j\in \mathcal J}$ is a countable system of nonempty mutually disjoint open intervals in $(0, \infty)$. Using \eqref{thm:PS:eq1}, we have
\begin{align*}
\int_E \phi(s)^q \dd s \leq \sum_{j \in \mathcal J} \int_{a_j}^{b_j} \phi(s)^q \dd s &\leq \sum_{j \in \mathcal J}
\int_{\{x\in\Sigma: u^*_\mu(b_j) < |u(x)| \leq u^*_\mu(a_j) \}} |\nabla u(x)|^q \dd \mu(x) \\
&= \sum_{j \in \mathcal J} \int_{\{x\in\Sigma: u^*_\mu(b_j) < |u(x)| < u^*_\mu(a_j) \}} |\nabla u(x)|^q \dd \mu(x).
\end{align*}
Note that the sets $\{x\in\Sigma: u^*_\mu(b_j) < |u(x)| < u^*_\mu(a_j) \}$, $j\in \mathcal J$, are mutually disjoint,
and $\mu(\{x\in\Sigma: u^*_\mu(b_j) < |u(x)| < u^*_\mu(a_j) \}) \leq b_j - a_j$. Combining that with the
Hardy-Littlewood inequality (\cite[Chapter~2, Theorem~2.2]{BS}), we obtain
\begin{align*}
\sum_{j \in \mathcal J} \int_{\{x\in\Sigma: u^*_\mu(b_j) < |u(x)| < u^*_\mu(a_j) \}} |\nabla u(x)|^q \dd \mu(x) &=
\int_{\bigcup_{j\in\mathcal J}\{x\in\Sigma: u^*_\mu(b_j) < |u(x)| < u^*_\mu(a_j) \}} |\nabla u(x)|^q \dd \mu(x) \\
&\leq \int_0^{\sum_{j \in \mathcal J}(b_j - a_j)} (\nabla u)_\mu^*(s)^q \dd s \\
&= \int_0^t (\nabla u)_\mu^*(s)^q \dd s.
\end{align*}
Hence
\begin{equation*}
\int_E \phi(s)^q \dd s \leq \int_0^t (\nabla u)_\mu^*(s)^q \dd s.
\end{equation*}

Now, note that the function $(0, \infty) \ni s \mapsto s^{\frac{q}{p} - 1}$ is nonincreasing. Therefore, it follows from \eqref{thm:PS:eq2} combined with the Hardy lemma \cite[Chapter~2, Proposition~3.6]{BS} that
\begin{equation}\label{thm:PS:eq3}
\int_0^\infty \phi^*(t)^q t^{\frac{q}{p} - 1}\dd t \leq \int_0^\infty (\nabla u)_\mu^*(t)^q t^{\frac{q}{p} - 1}\dd t.
\end{equation}
Note that
\begin{equation*}
|\nabla u_\mu^\bigstar|(x) = (-u_\mu^*)'(C_D |x|^D) D C_D |x|^{D-1} = \phi(C_D |x|^D) \quad \text{for a.e.~$x\in\Rde$}.
\end{equation*}
Hence
\begin{equation}\label{thm:PS:eq4}
(\nabla u_\mu^\bigstar)_\mu^*(t) = \phi^*(t) \quad \text{for every $t\in(0, \infty)$}.
\end{equation}
Therefore, \eqref{thm:PS:ineq} follows from \eqref{thm:PS:eq3} and \eqref{thm:PS:eq4}.

It remains to prove that $u_\mu^\bigstar \in V^1L^{p,q}(\Sigma, \mu)$. Since $u$ has compact support in $\Rde$, both $u_\mu^*$ and  $(-u_\mu^*)'$ vanish outside $[0, L]$, where $L = \mu(\spt u \cap \Sigma)$. Truncating the function $(-u_\mu^*)'$ if necessary, we may assume that $(-u_\mu^*)' \in L^\infty(0, \infty)$. Therefore, we have $(-u_\mu^*)' \in L^r(0, \infty)$ for every $r\in [1, \infty]$. Fix (any) finite $r > p$, and note that, using the embedding relations between Lorentz spaces and Lebesgue spaces (see~\cite{PKJF:13}), we have
\begin{equation*}
\|t^{\frac1{p} - \frac1{q}} f^*(t)\|_{L^q(0, 2L)} \leq C_E \|f\|_{L^r(0, 2L)}
\end{equation*}
for every function $f\in L^r(0,2L)$, vanishing outside $(0,2L)$, where the constant $C_E$ is independent of $f$. By the standard mollification argument, there are smooth nonnegative functions $\{\psi_n\}_{n = 1}^\infty\subseteq C_c^\infty(\Re)$ vanishing in $(2L, \infty)$ such that $\lim_{n\to\infty}\|(-u_\mu^*)' - \psi_n\|_{L^r(0, 2L)} = 0$.

Finally, for every $n\in\Ne$, we define the function $u_n$ as
\begin{equation*}
u_n(x) = \int_{C_D |x|^D}^\infty \psi_n(t) \dd t,\ x\in\Rde.
\end{equation*}
Note that $u_n\in \mathcal C_c^1(\Rde)$. We claim that $\lim_{n\to\infty}\|\nabla u_\mu^\bigstar - \nabla u_n\|_{p,q,\mu} = 0$. Indeed, setting $g_n(t) = t^\frac{D-1}{D}((-u_\mu^*)'(t) - \psi_n(t))$ and using \eqref{thm:PS:eq4}, we have
\begin{align*}
\|\nabla u_\mu^\bigstar - \nabla u_n\|_{p,q,\mu} &= DC_D^\frac1{D}\|t^{\frac1{p} - \frac1{q}} g_n^*(t)\|_{L^q(0,2L)} \\
&\leq DC_D^\frac1{D} (2L)^{\frac{D - 1}{D}}\|t^{\frac1{p} - \frac1{q}} \big( (-u_\mu^*)' - \psi_n \big)^*(t)\|_{L^q(0,2L)} \\
&\leq DC_D^\frac1{D} (2L)^{\frac{D - 1}{D}} C_E \|(-u_\mu^*)' - \psi_n\|_{L^r(0, 2L)}.
\end{align*}
Since the right-hand side goes to $0$ as $n\to\infty$, we obtain $u_\mu^\bigstar \in V^1L^{p,q}(\Sigma, \mu)$.
\end{proof}

\begin{rem}\label{rem:approximation_sequence}
For future reference, note that the approximation sequence $\{u_n\}_{n = 1}^\infty \subseteq C_c^1(\Rde)$ of radially nonincreasing functions from the preceding proof satisfies not only $\lim_{n\to\infty}\|\nabla u_\mu^\bigstar - \nabla u_n\|_{p,q,\mu}$ but also $\lim_{n\to\infty} \|u_\mu^\bigstar  - u_n\|_{s,q,\mu}$ for every $s < \infty$. Indeed, using the same notation as in the proof, we have
\begin{align*}
\|u_\mu^\bigstar  - u_n\|_{s,q,\mu} &\leq \Big\| t^{\frac1{s} - \frac1{q}} \int_t^{2L} \big| (-u_\mu^*)'(\tau) - \psi_n(\tau) \big| \dd \tau \, \chi_{(0, 2L)}(t) \Big\|_{L^q(0, \infty)} \\
&\leq \int_0^{2L} \big| (-u_\mu^*)'(\tau) - \psi_n(\tau) \big| \dd \tau \, \big\| t^{\frac1{s} - \frac1{q}}\chi_{(0, 2L)}(t) \big\|_{L^q(0, \infty)} \\
&\leq \|(-u_\mu^*)' - \psi_n\|_{L^r(0, 2L)} (2L)^{1 - \frac1{r}} \big\| t^{\frac1{s} - \frac1{q}}\chi_{(0, 2L)}(t) \big\|_{L^q(0, \infty)},
\end{align*}
where the right-hand side goes to $0$ as $n\to \infty$.
\end{rem}

Now, we are finally in a position to prove the optimal Sobolev-Lorentz inequality in $\Sigma$.
\begin{thm}\label{thm:Sob-Lor_emb}
Let $1\leq q\leq p < D$. Set $p^* = Dp/(D-p)$. We have
\begin{equation}\label{thm:Sob-Lor_emb:ineq}
\|u\|_{p^*, q,\mu} \leq \frac{p}{(D - p)\mu(B_1\cap \Sigma)^\frac1{D}} \|\nabla u\|_{p,q,\mu}
\quad \text{for every $u\in \mathcal C^1_c(\Rde)$}.
\end{equation}
\end{thm}
\begin{proof}
First, recall that the following Hardy inequality, which holds for every nonnegative measurable function $f$ on $(0, \infty)$ (see \cite[Theorem~330, p.~245]{HLP52} or, e.g., \cite[Section~1.3.1]{Mabook}, \cite{M:72}):
\begin{equation}\label{thm:Sob-Lor_emb:Hardy}
\Big( \int_0^\infty \Big(t^{\frac1{p^*} - \frac1{q}} \int_t^\infty f(s) \dd s\Big)^q \dd t
\Big)^\frac1{q} \leq p^* \Big( \int_0^\infty \Big(t^{1 + \frac1{p^*} - \frac1{q}} f(t)\Big)^q \dd t \Big)^\frac1{q}.
\end{equation}

Second, using \eqref{thm:PS:u_bigstar_as_integral} and \eqref{thm:Sob-Lor_emb:Hardy}, we have
\begin{align*}
\|u\|_{p^*,q,\mu} &=  \|u_\mu^\bigstar\|_{p^*,q,\mu} = \Big\| t^{\frac1{p^*} - \frac1{q}} \int_t^\infty (-u_\mu^*)'(s) \dd s \Big\|_{L^q(0, \infty)} \\
&\leq p^* \Big\| t^{1 + \frac1{p^*} - \frac1{q}} (-u_\mu^*)'(t) \Big\|_{L^q(0, \infty)} = p^* \Big\| t^{\frac1{p^*} - \frac1{q} + \frac1{D}} ( -u_\mu^*)'(t) t^{\frac{D - 1}{D}} \Big\|_{L^q(0, \infty)} \\
&= p^* \Big\| t^{\frac1{p} - \frac1{q}} (-u_\mu^*)'(t) t^{\frac{D - 1}{D}} \Big\|_{L^q(0, \infty)}
\end{align*}
for every $u\in\mathcal C_c^1(\Rde)$. Furthermore, thanks to the Hardy-Littlewood inequality (\cite[Chapter~2, Theorem~2.2]{BS}), we have
\begin{equation*}
\Big\| t^{\frac1{p} - \frac1{q}} (-u_\mu^*)'(t) t^{\frac{D - 1}{D}} \Big\|_{L^q(0, \infty)} \leq D^{-1} C_D^{-\frac1{D}} \Big\| t^{\frac1{p} - \frac1{q}} \phi^*(t) \Big\|_{L^q(0, \infty)},
\end{equation*}
where the function $\phi$ is defined by \eqref{thm:PS:def_phi}. Hence
\begin{equation}\label{thm:Sob-Lor_emb:Sob_with_phi}
\|u\|_{p^*,q,\mu} \leq p^*D^{-1} C_D^{-\frac1{D}} \Big\| t^{\frac1{p} - \frac1{q}} \phi^*(t) \Big\|_{L^q(0, \infty)}
\end{equation}
for every $u\in\mathcal C_c^1(\Rde)$.

Finally, combining \eqref{thm:Sob-Lor_emb:Sob_with_phi} with \eqref{thm:PS:eq4}, we obtain
\begin{equation*}
\|u\|_{p^*,q,\mu} \leq p^* D^{-1} C_D^{-\frac1{D}} \|\nabla u_\mu^\bigstar\|_{p,q,\mu}
\end{equation*}
for every $u\in\mathcal C_c^1(\Rde)$. Hence, combining that with \eqref{thm:PS:ineq}, we have
\begin{equation*}
\|u\|_{p^*,q,\mu} \leq p^* D^{-1} C_D^{-\frac1{D}} \|\nabla u\|_{p,q,\mu} \quad \text{for every $u\in\mathcal C_c^1(\Rde)$}. \qedhere
\end{equation*}
\end{proof}

\begin{rem}\label{rem:optimal_constant}
By straightforwardly modifying the maximizing sequence introduced by Alvino in \cite{A:77} (cf.~\cite{CRT:18}), it is not hard to prove that the constant in \eqref{thm:Sob-Lor_emb:ineq} is optimal. In other words,
\begin{equation*}
\|E\| = \frac{p}{(D - p)\mu(B_1\cap \Sigma)^\frac1{D}},
\end{equation*}
where $\|E\|$ is the (operator) norm of the embedding operator $E\colon V^1L^{p,q}(\Sigma, \mu) \to L^{p^*, q}(\Sigma, \mu)$.
\end{rem}


\section{Quantitative aspects of the non-compactness}\label{sec:bern}

We start with two auxiliary propositions.
\begin{prop}\label{prop:extremal_functions_constant_near_zero_radially_symmetric}
Let $1\leq q\leq p < D$. Let $E\colon V^1L^{p,q}(\Sigma, \mu) \to L^{p^*, q}(\Sigma, \mu)$ be
the embedding operator, where $p^*$ is as in Theorem~\ref{thm:Sob-Lor_emb}. For $0 < r < R$, set
\begin{equation*}
\mathcal F_{r,R}  = \{u\in\mathcal C^1_c(\Rde): u=u_\mu^\bigstar, \spt u \subseteq B_R, \nabla u\equiv0\ \text{in}\ B_r\}.
\end{equation*}
Then, for every $R > 0$,
\begin{equation}\label{prop:extremal_functions_constant_near_zero_radially_symmetric_equality}
\|E\| = \sup_{\substack{u\in \mathcal F_{r,R} \\ r\in(0, R)}} \frac{\|u\|_{p^*,q,\mu}}{\|\nabla u\|_{p,q,\mu}}.
\end{equation}
\end{prop}
\begin{proof}
In view of Theorem~\ref{thm:PS},
\begin{equation*}
\|E\| = \sup_{u\in \mathcal C^1_c(\Rde)} \frac{\|u_\mu^\bigstar\|_{p^*,q,\mu}}{\|\nabla u_\mu^\bigstar\|_{p,q,\mu}}.
\end{equation*}
First, we claim that
\begin{equation}\label{prop:extremal_functions_constant_near_zero_radially_symmetric:sup_only_radially_symmetric_compactly_supported}
\|E\| = \sup_{\substack{u\in \mathcal C^1_c(\Rde) \\ \spt u_\mu^\bigstar\subseteq B_R}} \frac{\|u_\mu^\bigstar\|_{p^*,q,\mu}}{\|\nabla u_\mu^\bigstar\|_{p,q,\mu}}.
\end{equation}
To that end, it is  sufficient to prove that, for every $u\in \mathcal C^1_c(\Rde)$,
\begin{equation}\label{prop:extremal_functions_constant_near_zero_radially_symmetric:eq1}
\frac{\|u_\mu^\bigstar\|_{p^*,q,\mu}}{\|\nabla u_\mu^\bigstar\|_{p,q,\mu}} \leq \sup_{\substack{v\in
\mathcal C^1_c(\Rde) \\ \spt v_\mu^\bigstar\subseteq B_R}}
\frac{\|v_\mu^\bigstar\|_{p^*,q,\mu}}{\|\nabla v_\mu^\bigstar\|_{p,q,\mu}}.
\end{equation}
Let $u\in \mathcal C^1_c(\Rde)$. Since $u$ is compactly supported in $\Rde$, $u_\mu^\bigstar$ is supported in $B_{\tilde{R}}$ for some $\tilde{R}>0$. Set $\kappa = \tilde{R}/R$ and let $u_\kappa$ be the function defined as
\begin{equation*}
u_{\kappa}(x) = u(\kappa x),\ x\in\Rde.
\end{equation*}
Clearly $u_{\kappa}\in \mathcal C^1_c(\Rde)$. We have
\begin{equation}\label{prop:extremal_functions_constant_near_zero_radially_symmetric:eq2}
(u_{\kappa})_\mu^*(t) = u_\mu^*(\kappa^D t) \quad \text{for every $t>0$}
\end{equation}
thanks to the homogeneity of $w$. Indeed,
\begin{align*}
  (u_{\kappa})_{*\mu}(t)
  &=
  \int_{\{x\in \Sigma:|u_{\kappa}(x)|>t\}}w(x)\dd x
  =
  \int_{\{x\in \Sigma:|u(\kappa x)|>t\}}w(x)\dd x \\
	&=
	\kappa^{-d}\int_{\{x\in \Sigma:|u(x)|>t\}}w(x/\kappa)\dd x
	=
  \kappa^{-d-\al} \int_{\{x\in \Sigma:|u(x)|>t\}}w(x)\dd x \\
	&=
	\kappa^{-D}u_{*\mu}(t)
\end{align*}
for every $t>0$, where we used \eqref{Wfunc} in the second to last equality. Hence \eqref{prop:extremal_functions_constant_near_zero_radially_symmetric:eq2} is true. Similarly, it is easy to see that
\begin{equation}\label{prop:extremal_functions_constant_near_zero_radially_symmetric:eq3}
(\nabla (u_{\kappa})_\mu^\bigstar)_\mu^*(t) = \kappa(\nabla u_\mu^\bigstar)_\mu^*(\kappa^D t) \quad \text{for every $t>0$}.
\end{equation}
By plugging \eqref{prop:extremal_functions_constant_near_zero_radially_symmetric:eq2} and \eqref{prop:extremal_functions_constant_near_zero_radially_symmetric:eq3} in the definition of $\|\cdot\|_{p^*,q,\mu}$ and $\|\cdot\|_{p,q,\mu}$, respectively, we observe that
\begin{equation*}
\frac{\|(u_{\kappa})_\mu^\bigstar\|_{p^*,q,\mu}}{\|\nabla (u_{\kappa})_\mu^\bigstar\|_{p,q,\mu}} = \frac{\kappa^{-\frac{D}{p^*}}\|u_\mu^\bigstar\|_{p^*,q,\mu}}{\kappa^{1 - \frac{D}{p}}\|\nabla u_\mu^\bigstar\|_{p,q,\mu}} = \frac{\|u_\mu^\bigstar\|_{p^*,q,\mu}}{\|\nabla u_\mu^\bigstar\|_{p,q,\mu}}.
\end{equation*}
Furthermore, note that \eqref{prop:extremal_functions_constant_near_zero_radially_symmetric:eq2} implies that
\begin{equation*}
\spt (u_{\kappa})_\mu^\bigstar \subseteq B_{\tilde{R}/\kappa} = B_{R}.
\end{equation*}
Hence \eqref{prop:extremal_functions_constant_near_zero_radially_symmetric:eq1} follows from the last two observations.

Next, we claim that
\begin{equation}\label{prop:extremal_functions_constant_near_zero_radially_symmetric:sup_only_radially_symmetric_compactly_supported_constant_near_zero}
\sup_{\substack{u\in \mathcal C^1_c(\Rde) \\ \spt u_\mu^\bigstar\subseteq B_R}} \frac{\|u_\mu^\bigstar\|_{p^*,q,\mu}}{\|\nabla u_\mu^\bigstar\|_{p,q,\mu}} = \sup_{\substack{u\in \mathcal F_{r,R} \\ r\in(0, R)}} \frac{\|u\|_{p^*,q,\mu}}{\|\nabla u\|_{p,q,\mu}}.
\end{equation}
To that end, it is sufficient to show that
\begin{equation}\label{prop:extremal_functions_constant_near_zero_radially_symmetric:eq4}
\sup_{\substack{u\in \mathcal C^1_c(\Rde) \\ \spt u_\mu^\bigstar\subseteq B_R}} \frac{\|u_\mu^\bigstar\|_{p^*,q,\mu}}{\|\nabla u_\mu^\bigstar\|_{p,q,\mu}} \leq \sup_{\substack{u\in \mathcal F_{r,R} \\ r\in(0, R)}} \frac{\|u\|_{p^*,q,\mu}}{\|\nabla u\|_{p,q,\mu}}.
\end{equation}
Let $u\in \mathcal C^1_c(\Rde)$ be such that $\spt u_\mu^\bigstar\subseteq B_R$. Thanks to \eqref{thm:PS:u_bigstar_as_integral}, we have
\begin{equation*}
u^\bigstar_\mu(x) = \int_{C_D|x|^D}^\infty (-u_\mu^*)'(t) \dd t \quad \text{for every $x\in\Rde$}.
\end{equation*}
Approximating $u^\bigstar_\mu$ as in the proof of Theorem~\ref{thm:PS} (see also Remark~\ref{rem:approximation_sequence}) if necessary, we may assume that $u^\bigstar_\mu\in \mathcal C^1_c(\Rde)$. Let $\eta_n\in \mathcal C(0,\infty)$, $n\in\Ne$, be a sequence of cutoff functions such that
\begin{equation*}
0\leq \eta_n\leq1,\quad \eta_n \equiv 0\ \text{in $\Big(0, \frac1{n+1}\Big]$},\quad\text{and } \eta_n \equiv 1 \ \text{in $\Big[\frac1{n}, \infty \Big)$}.
\end{equation*}
Let $u_n$, $n\in\Ne$, be the sequence of functions defined as
\begin{equation*}
u_n(x) = \int_{C_D|x|^D}^\infty (-u_\mu^*)'(t)\eta_n(t) \dd t,\ x\in\Rde.
\end{equation*}
Then $u_n\in\mathcal C_c^1(\Rde)$ and $u_n=(u_n)_\mu^\bigstar$. Note that each function $u_n$ is constant in a ball 
centered at the origin. Hence $\nabla u_n\equiv0$ in this ball. Furthermore, it is easy to see that
\begin{equation}\label{prop:extremal_functions_constant_near_zero_radially_symmetric:eq5}
\lim_{n\to\infty} \|u_n\|_{p^*,q,\mu} = \|u_\mu^\bigstar\|_{p^*,q,\mu}
\end{equation}
and
\begin{equation}\label{prop:extremal_functions_constant_near_zero_radially_symmetric:eq6}
\lim_{n\to\infty} \|\nabla u_n\|_{p,q,\mu} = \|\nabla u_\mu^\bigstar\|_{p,q,\mu}.
\end{equation}
Combining \eqref{prop:extremal_functions_constant_near_zero_radially_symmetric:eq5} and \eqref{prop:extremal_functions_constant_near_zero_radially_symmetric:eq6} with our observations about the functions $u_n$, we obtain \eqref{prop:extremal_functions_constant_near_zero_radially_symmetric:eq4}.

Finally, \eqref{prop:extremal_functions_constant_near_zero_radially_symmetric_equality} follows from \eqref{prop:extremal_functions_constant_near_zero_radially_symmetric:sup_only_radially_symmetric_compactly_supported} and \eqref{prop:extremal_functions_constant_near_zero_radially_symmetric:sup_only_radially_symmetric_compactly_supported_constant_near_zero}.
\end{proof}

\begin{prop}\label{prop:almost_extremal_system}
Let $1\leq q\leq p < D$. Let $E\colon V^1L^{p,q}(\Sigma, \mu) \to L^{p^*, q}(\Sigma, \mu)$ be the embedding operator, where $p^*$ is as in Theorem~\ref{thm:Sob-Lor_emb}. For every $0 < \lambda < \|E\|$, $\varepsilon_1 > 0$, and $\varepsilon_2 > 0$, there is a sequence of functions $\{u_j\}_{j = 1}^\infty\subseteq \mathcal C^1_c(\Rde)$ such that
\begin{enumerate}[(i)]
	\item $\|u_j\|_{p^*,q,\mu} = \lambda$ and $\|\nabla u_j\|_{p,q,\mu} = 1$ for every $j\in\Ne$.
	\item $\spt u_{j+1}\subsetneq \spt u_j$ and $\spt \nabla u_j \subseteq \spt u_j\setminus \spt u_{j + 1}$ for every $j\in\Ne$.
	\item $u_j = (u_j)_\mu^\bigstar$ for every $j\in\Ne$.
		\item $\spt u_j \to \emptyset$ as $j \to \infty$.
	\item For every sequence $\{\alpha_j\}_{j = 1}^\infty$ we have
	\begin{equation}\label{prop:almost_extremal_system_superadditivity}
	\Big\| \sum_{j = 1}^\infty \alpha_j u_j \Big\|_{p^*,q,\mu} \geq \Bigg( \frac{\lambda}{1 + \varepsilon_1} - \varepsilon_2 \Bigg)\Big( \sum_{j = 1}^\infty |\alpha_j|^q \Big)^{\frac1{q}}.
	\end{equation}
\end{enumerate}
\end{prop}
\begin{proof}
We construct the desired sequence of functions inductively. Fix $0 < \lambda < \|E\|$, $\varepsilon_1 > 0$, and $\varepsilon_2 > 0$. If $q\in(1, p]$, let $a\in(0,1)$ be so small that
\begin{equation*}
\frac{a^{q'}}{1-a^{q'}} \leq \varepsilon_2^{q'}.
\end{equation*}
For $j\in\Ne$, set
\begin{equation*}
\gamma_j = \Big\{
  \begin{array}{ll}\vspace{4pt}
    \varepsilon_2 \quad & \hbox{if $q = 1$,} \\
    a^j  & \hbox{if $q\in(1, p]$.}
  \end{array}
\end{equation*}
Note that
\begin{equation}\label{prop:almost_extremal_system_ell_q_prime_norm_of_gamma_j}
\|\{\gamma_j\}_{j = 1}^\infty\|_{\ell_{q'}} = \Big( \sum_{j = 1}^\infty |\gamma_j|^{q'} \Big)^\frac1{q'} \leq \varepsilon_2.
\end{equation}

First, take any $R_1 > 0$. Thanks to Proposition~\ref{prop:extremal_functions_constant_near_zero_radially_symmetric}, 
there is $r_1\in(0, R_1)$ and a function $u_1 = (u_1)_\mu^\bigstar \in \mathcal C_c^1(\Rde)$ such that
\begin{equation*}
\|u_1\|_{p^*,q,\mu} = \lambda,\ \|\nabla u_1\|_{p,q,\mu} = 1,\ \spt u_1 \subseteq B_{R_1},\ \spt \nabla u_1 \subseteq B_{R_1}\setminus \overline{B}_{r_1}.
\end{equation*}
Set $\delta_1 = \mu(B_{R_1} \cap \Sigma)$. Using the absolute continuity of the Lorentz norm $\|\cdot\|_{p^*,q,\mu}$, we can find $\tilde{R}_2\in(0, r_1)$ such that
\begin{align}
\|u_1\chi_{B_{\tilde{R}_2}}\|_{p^*,q,\mu} &\leq \gamma_1 \label{prop:almost_extremal_system_first_step_u_one} \\
\intertext{and}
(1 + \varepsilon_1)\|u_1\chi_{B_{R_1}\setminus B_{\tilde{R}_2}}\|_{p^*,q,\mu} &\geq \|u_1\|_{p^*,q,\mu} \notag.
\end{align}
Furthermore, by the dominated convergence theorem combined with the last inequality, we can find $R_2 \in(0, \tilde{R}_2)$ so small that
\begin{equation}\label{prop:almost_extremal_system_first_step_u_one_tilde}
(1 + \varepsilon_1)^q \int_{\delta_2}^{\delta_1} t^{\frac{q}{p^*} - 1} (u_1\chi_{B_{R_1}\setminus B_{\tilde{R}_2}})_\mu^*(t)^q \dd t \geq \|u_1\|_{p^*,q,\mu}^q,
\end{equation}
where $\delta_2 = \mu(B_{R_2} \cap \Sigma) \in (0, \delta_1)$. Moreover, $R_2$ can be taken so small that $\overline{B}_{R_2} \subsetneq \spt u_1$. Note that both \eqref{prop:almost_extremal_system_first_step_u_one} and \eqref{prop:almost_extremal_system_first_step_u_one_tilde} are still valid with $\tilde{R}_2$ replaced by $R_2$.

Second, let $m\in\Ne$, and assume that we have already found $\{u_j\}_{j = 1}^m$, $\{\delta_j\}_{j = 1}^{m + 1}$, $\{r_j\}_{j = 1}^{m}$, and $\{R_j\}_{j = 1}^{m + 1}$ such that
\begin{align}
&\|u_j\|_{p^*,q,\mu} = \lambda,\ \|\nabla u_j\|_{p,q,\mu} = 1,\ u_j = (u_j)_\mu^\bigstar \in \mathcal C_c^1(\Rde), \notag\\
&\overline{B}_{R_{j+1}} \subsetneq \spt u_j \subseteq B_{R_j},\ \spt \nabla u_j \subseteq B_{R_j}\setminus \overline{B}_{r_j}, \notag\\
&0< R_{j + 1} < r_j < R_j,\ \delta_j = \mu(B_{R_j} \cap \Sigma),\ \delta_{j+1} \in (0, \delta_j/j), \notag\\
&\|u_j\chi_{B_{R_{j + 1}}}\|_{p^*,q,\mu} \leq \gamma_j, \label{prop:almost_extremal_system_inductive_step_hyp_u_j} \\
\intertext{and}
&(1 + \varepsilon_1)^q \int_{\delta_{j+1}}^{\delta_j} t^{\frac{q}{p^*} - 1} (u_1\chi_{B_{R_j}\setminus B_{R_{j+1}}})_\mu^*(t)^q \dd t \geq \|u_j\|_{p^*,q,\mu}^q \label{prop:almost_extremal_system_inductive_step_hyp_u_j_tilde}
\end{align}
for every $j = 1, \dots, m$. The inductive step is very similar to the first step. Thanks to Proposition~\ref{prop:extremal_functions_constant_near_zero_radially_symmetric} with $R = R_{m+1}$, there is $r_{m + 1}\in(0, R_{m + 1})$ and a function $u_{m + 1} = (u_{m + 1})_\mu^\bigstar \in \mathcal C_c^1(\Rde)$ such that $\|u_{m + 1}\|_{p^*,q,\mu} = \lambda$, $\|\nabla u_{m + 1}\|_{p,q,\mu} = 1$, $\spt u_{m+1} \subseteq B_{R_{m + 1}}$, and $\spt \nabla u_{m + 1} \subseteq B_{R_{m + 1}}\setminus \overline{B}_{r_{m + 1}}$. Now, we find $\tilde{R}_{m + 2}\in(0, r_{m + 1})$ such that
\begin{align*}
\|u_{m + 1}\chi_{B_{\tilde{R}_{m + 2}}}\|_{p^*,q,\mu} &\leq \gamma_{m+1}\\
\intertext{and}
(1 + \varepsilon_1)\|u_{m + 1}\chi_{B_{R_{m + 1}}\setminus B_{\tilde{R}_{m + 2}}}\|_{p^*,q,\mu} &\geq \|u_{m + 1}\|_{p^*,q,\mu} \notag.
\end{align*}
Next, we find $\delta_{m + 2}\in (0, \delta_{m+1}/(m+1))$ so small that $R_{m + 2}$ defined by $\delta_{m + 2} = \mu(B_{R_{m+2}} \cap \Sigma)$ satisfies $R_{m + 2} \in(0, \tilde{R}_{m + 2})$, $\overline{B}_{R_{m+2}} \subsetneq \spt u_{m+1}$, and we have
\begin{equation*}
(1 + \varepsilon_1)^q \int_{\delta_{m+2}}^{\delta_{m+1}} t^{\frac{q}{p^*} - 1} (u_{m+1}\chi_{B_{R_{m+1}}\setminus B_{\tilde{R}_{m+2}}})_\mu^*(t)^q \dd t \geq \|u_{m + 1}\|_{p^*,q,\mu}^q.
\end{equation*}
Finally, we observe that the last three inequalities still hold with $\tilde{R}_{m+2}$ replaced by $R_{m+2}$. This finishes the construction of the desired sequence.

Next, we can easily verify that the constructed 
sequence $\{u_j\}_{j=1}^\infty$ fulfills properties (i) to (iv). 
However, we still need to prove that the fifth property is also satisfied. Fix $\{\alpha_j\}_{j = 1}^\infty \subseteq \Re$. Clearly, we may assume, without loss of generality, that the left-hand side of \eqref{prop:almost_extremal_system_superadditivity} is finite. Then
\begin{align}
\Big\| \sum_{j = 1}^\infty \alpha_j u_j \Big\|_{p^*,q,\mu} &\geq \Big\| 
\sum_{j = 1}^\infty \alpha_j \tilde{u}_j \Big\|_{p^*,q,\mu} - \Big\| 
\sum_{j = 1}^\infty \alpha_j (u_j - \tilde{u}_j) \Big\|_{p^*,q,\mu} \notag\\
&\geq \Big\| \sum_{j = 1}^\infty \alpha_j \tilde{u}_j \Big\|_{p^*,q,\mu} - \Big\| 
\sum_{j = 1}^\infty \alpha_j u_j\chi_{B_{R_{j+1}}} \Big\|_{p^*,q,\mu} \label{prop:almost_extremal_system_superadditivity_lower_bound},
\end{align}
where the functions $\tilde u_j$ are defined as
\begin{equation*}
\tilde u_j = u_j\chi_{B_{R_j}\setminus B_{R_{j+1}}},\ j\in\Ne.
\end{equation*}

As for the first term, since the functions $\tilde u_j$ have mutually disjoint supports, 
it follows (e.g., see~\cite[Inequality (3.5)]{LaMi23}) that
\begin{equation*}
\Big( \sum_{j = 1}^\infty \alpha_j \tilde{u}_j \Big)_\mu^* \geq \sum_{j = 1}^\infty |\alpha_j|(\tilde{u}_j)_\mu^*\chi_{(\delta_{j+1}, \delta_j)}.
\end{equation*}
Combining that with \eqref{prop:almost_extremal_system_inductive_step_hyp_u_j_tilde} and with the fact that the intervals $\{(\delta_{j+1}, \delta_j)\}_{j = 1}^\infty$ are mutually disjoint, we obtain
\begin{align*}
\Big\| \sum_{j = 1}^\infty \alpha_j \tilde{u}_j \Big\|_{p^*,q,\mu}^q \geq  \sum_{j = 1}^\infty |\alpha_j|^q \int_{\delta_{j+1}}^{\delta_j} t^{\frac{q}{p^*} - 1} (\tilde{u}_j)_\mu^*(t)^q \dd t \geq \sum_{j = 1}^\infty \frac{\|u_j\|_{p^*,q,\mu}^q}{(1 + \varepsilon_1)^q}|\alpha_j|^q.
\end{align*}
Hence, since $\|u_j\|_{p^*,q,\mu} = \lambda$ for every $j\in\Ne$,
\begin{equation}\label{prop:almost_extremal_system_superadditivity_lower_bound_first_term}
\Big\| \sum_{j = 1}^\infty \alpha_j \tilde{u}_j \Big\|_{p^*,q,\mu} \geq \frac{\lambda}{1 + \varepsilon_1}\Big( \sum_{j = 1}^\infty |\alpha_j|^q \Big)^{\frac1{q}}.
\end{equation}
Concerning the second term, we use \eqref{prop:almost_extremal_system_inductive_step_hyp_u_j} and the H\"older inequality to obtain
\begin{align*}
\Big\| \sum_{j = 1}^\infty \alpha_j u_j\chi_{B_{R_{j+1}}} \Big\|_{p^*,q,\mu} &\leq \sum_{j = 1}^\infty |\alpha_j| \|u_j\chi_{B_{R_{j+1}}}\|_{p^*,q,\mu} \leq \sum_{j = 1}^\infty |\alpha_j| \gamma_j \\
&\leq  \|\{\alpha_j\}_{j = 1}^\infty\|_{\ell_q} \|\{\gamma_j\}_{j = 1}^\infty\|_{\ell_{q'}}.
\end{align*}
Combining that with \eqref{prop:almost_extremal_system_ell_q_prime_norm_of_gamma_j}, we arrive at
\begin{equation}\label{prop:almost_extremal_system_superadditivity_lower_bound_second_term}
\Big\| \sum_{j = 1}^\infty \alpha_j u_j\chi_{B_{R_{j+1}}} \Big\|_{p^*,q,\mu} \leq \varepsilon_2 \|\{\alpha_j\}_{j = 1}^\infty\|_{\ell_q}.
\end{equation}
Finally, by combining \eqref{prop:almost_extremal_system_superadditivity_lower_bound_first_term} and 
\eqref{prop:almost_extremal_system_superadditivity_lower_bound_second_term} with 
\eqref{prop:almost_extremal_system_superadditivity_lower_bound}, we obtain 
\eqref{prop:almost_extremal_system_superadditivity}, which finishes the proof.
\end{proof}

Now, we can show that all Bernstein numbers of the embedding operator
$E\colon V^1L^{p,q}(\Sigma, \mu) \to L^{p^*, q}(\Sigma, \mu)$ as well as its measure of non-compactness coincide with its norm.
\begin{thm}\label{thm:bernstein_numbers}
Let $1\leq q\leq p < D$. Let $E\colon V^1L^{p,q}(\Sigma, \mu) \to L^{p^*, q}(\Sigma, \mu)$ be the embedding operator, where $p^*$ is as in Theorem~\ref{thm:Sob-Lor_emb}. Then
\begin{equation}\label{thm:bernstein_numbers_value}
b_m(E) = \beta(E) = \|E\| \quad \text{for every $m\in\Ne$}.
\end{equation}
In particular, $E$  is maximally non-compact. Furthermore, $E$ is not strictly singular.
\end{thm}
\begin{proof}
In view of the property (S1) of (strict) $s$-numbers, in order to prove that $b_m(E) = \|E\|$ for every $m\in\Ne$, it is sufficient to show that
\begin{equation*}
b_m(E) \geq \|E\| \quad \text{for every $m\in\Ne$}.
\end{equation*}
Fix arbitrary $0 < \lambda < \|E\|$, $\varepsilon_1 > 0$, and $\varepsilon_2 > 0$. Let $\{u_j\}_{j = 1}^\infty\subseteq \mathcal C_c^1(\Rde)$ be the sequence of functions whose existence is guaranteed by Proposition~\ref{prop:almost_extremal_system}.

Fix $m\in\Ne$. Let $X_m = \spn\{u_1, \dots,  u_m\}$. Since $\spt u_{j+1} \subsetneq \spt u_j$ for every $j\in\Ne$, $X_m$ is an $m$-dimensional subspace of $V^1L^{p,q}(\Sigma, \mu)$. Hence
\begin{equation}\label{thm:bernstein_numbers_lower_bound}
b_m(E) \geq \inf_{u\in X_m\setminus\{0\}} \frac{\|u\|_{p^*,q,\mu}}{\|\nabla u\|_{p,q,\mu}}.
\end{equation}

Since the functions $\nabla u_j$ have disjoint supports, we have
\begin{equation*}
\Big(\sum_{j = 1}^m \alpha_j \nabla u_j \Big)_{*\mu}(t) = \sum_{j = 1}^m (\nabla u_j)_{*\mu}\Big( \frac{t}{|\alpha_j|} \Big) \quad \text{for every $t>0$}
\end{equation*}
and every $\{\alpha_j\}_{j=1}^\infty\subseteq \Re$ (when $\alpha_j = 0$, $(\nabla u_j)_{*\mu}(t/|\alpha_j|)$ is to be interpreted as $0$). Furthermore, since $q/p\in(0, 1]$, the function $[0, \infty)\ni a \mapsto a^{\frac{q}{p}}$ is subadditive, and so
\begin{equation*}
\Bigg( \sum_{j = 1}^m (\nabla u_j)_{*\mu}\Big( \frac{t}{|\alpha_j|} \Big) \Bigg)^{\frac{q}{p}} \leq \sum_{j = 1}^m (\nabla u_j)_{*\mu}\Big( \frac{t}{|\alpha_j|} \Big)^{\frac{q}{p}} \quad \text{for every $t>0$}
\end{equation*}
and every $\{\alpha_j\}_{j=1}^\infty\subseteq\Re$. Therefore, combining these two observations with \eqref{prel:Lorentz_norm_distributional_function}, we arrive at
\begin{align*}
\Big\| \sum_{j = 1}^m \alpha_j\nabla u_j \Big\|_{p,q,\mu}^q &\leq p\sum_{j = 1}^m \int_0^\infty t^{q - 1} (\nabla u_j)_{*\mu}\Big( \frac{t}{|\alpha_j|} \Big)^\frac{q}{p} \dd t \\
&= p\sum_{j = 1}^m |\alpha_j|^q \int_0^\infty t^{q - 1} (\nabla u_j)_{*\mu}(t)^\frac{q}{p} \dd t = \sum_{j = 1}^m |\alpha_j|^q \|\nabla u_j\|_{p,q,\mu}^q
\end{align*}
for every $\{\alpha_j\}_{j=1}^\infty\subseteq\Re$. Hence, combining this with the fact that $\|\nabla u_j\|_{p,q,\mu} = 1$ for every $j\in\Ne$, we have
\begin{equation}\label{thm:bernstein_numbers_upper_bound_denominator}
\Big\| \sum_{j = 1}^m \alpha_j\nabla u_j \Big\|_{p,q,\mu} \leq \Big(\sum_{j = 1}^\infty |a_j|^q\Big)^\frac1{q}
\end{equation}
for every $\{\alpha_j\}_{j=1}^\infty\subseteq\Re$.

Thanks to \eqref{prop:almost_extremal_system_superadditivity} with \eqref{thm:bernstein_numbers_upper_bound_denominator}, we obtain
\begin{equation*}
\frac{\|\sum_{j = 1}^m \alpha_j u_j\|_{p^*,q,\mu}}{\|\sum_{j = 1}^m \alpha_j \nabla u_j\|_{p,q,\mu}} \geq \Bigg( \frac{\lambda}{1 + \varepsilon_1} - \varepsilon_2 \Bigg)
\end{equation*}
for every $\{\alpha_j\}_{j = 1}^\infty \in \ell_q$. It follows that
\begin{equation}\label{thm:bernstein_numbers_lower_bound_Xm}
\inf_{u\in X_m\setminus\{0\}} \frac{\|u\|_{p^*,q,\mu}}{\|\nabla u\|_{p,q,\mu}} \geq \Bigg( \frac{\lambda}{1 + \varepsilon_1} - \varepsilon_2 \Bigg).
\end{equation}
Therefore, by combining \eqref{thm:bernstein_numbers_lower_bound} with \eqref{thm:bernstein_numbers_lower_bound_Xm}, we arrive at
\begin{equation*}
b_m(E) \geq \Bigg( \frac{\lambda}{1 + \varepsilon_1} - \varepsilon_2 \Bigg).
\end{equation*}
By letting $\varepsilon_2\to0^+$, $\varepsilon_1\to0^+$, and $\lambda\to \|E\|^-$, we obtain \eqref{thm:bernstein_numbers_value}.

Next, as for the fact that $E$ is not strictly singular, we consider the infinite-dimensional subspace $Z$ of $V^1L^{p,q}(\Sigma, \mu)$ defined as $Z = \spn\{u_j: j\in\Ne\}$. Note that
$u=\sum_{j = 1}^\infty \alpha_j u_j \in \mathcal C_c^1(\Rde)$ for every $\{\alpha_j\}_{j = 1}^\infty$ because the sum is locally finite and $\spt u \subseteq \spt u_1$. Using the same arguments as above, we obtain
\begin{equation*}
\inf_{u\in Z\setminus\{0\}} \frac{\|u\|_{p^*,q,\mu}}{\|\nabla u\|_{p,q,\mu}} \geq \|E\|.
\end{equation*}

Finally, in order to establish that $E$ is maximally non-compact, 
we must show that $\| E \| = \beta(E)$. Let us assume, for the sake of contradiction, that $\beta(E) < \| E \|$.
Choose any $\lambda \in (\beta(E), \| E \|)$, and consider the sequence of functions $\{u_j\}_{j=1}^\infty$ 
from Proposition~\ref{prop:almost_extremal_system}, with arbitrarily chosen $\varepsilon_1$ and 
$\varepsilon_2$ (whose specific values are irrelevant). Now, fix an $r \in (\beta(E), \lambda)$. Since $r > \beta(E)$, using the definition of the measure of non-compactness,
there are $m\in\Ne$ and functions $\{g_k\}_{k = 1}^m \subseteq L^{p^*,q}(\Sigma, \mu)$ such that
\begin{equation}\label{thm:bernstein_numbers_upper_bound_max_noncompact_finite_cover}
B_{V^1 L^{p, q}(\Sigma, \mu)} \subseteq \bigcup_{k = 1}^m \big( g_k + r B_{L^{p^*, q}(\Sigma, \mu)} \big).
\end{equation}
Set $\varepsilon_0 = \lambda - r > 0$.

Thanks to \eqref{thm:bernstein_numbers_upper_bound_max_noncompact_finite_cover}, for every $j\in\Ne$ there is $k_j\in\{1, \dots, m\}$ such that
\begin{equation}\label{thm:bernstein_numbers_upper_bound_max_noncompact_eq1}
\|u_j - g_{k_j}\|_{p^*,q,\mu} \leq r.
\end{equation}
Set
\begin{equation*}
h_j = g_{k_j}\chi_{\spt u_j} \quad \text{for every $j\in\Ne$}.
\end{equation*}
Since
\begin{equation*}
|h_j - u_j| \leq |g_{k_j} - u_j| \quad \text{for every $j\in\Ne$ and $\mu$-a.e.~in $\Sigma$},
\end{equation*}
it follows from \eqref{thm:bernstein_numbers_upper_bound_max_noncompact_eq1} that
\begin{equation*}
\|u_j - h_j\|_{p^*,q,\mu} \leq r \quad \text{for every $j\in\Ne$}.
\end{equation*}

Now, on the one hand,
\begin{equation}\label{thm:bernstein_numbers_upper_bound_max_noncompact_eq2}
\|h_j\|_{p^*,q,\mu} \geq	\|u_j\|_{p^*,q,\mu} - \|u_j - h_j\|_{p^*,q,\mu} \geq \lambda - r = \varepsilon_0 \quad \text{for every $j\in\Ne$}.
\end{equation}
On the other hand, since
\begin{align*}
\spt h_j \subseteq \spt u_j \to \emptyset \quad &\text{as $j\to\infty$} \\
\intertext{and}
|h_j| \leq \sum_{k = 1}^m |g_k| \in L^{p^*, q}(\Sigma, \mu) \quad &\text{for every $j\in\Ne$ and $\mu$-a.e.~in $\Sigma$},
\end{align*}
it follows from the absolute continuity of the $\|\cdot\|_{p^*,q,\mu}$ norm that
\begin{equation}\label{thm:bernstein_numbers_upper_bound_max_noncompact_eq3}
\lim_{j \to \infty} \|h_j\|_{p^*,q,\mu} = 0.
\end{equation}
However, a contradiction arises when considering 
both \eqref{thm:bernstein_numbers_upper_bound_max_noncompact_eq2} 
and \eqref{thm:bernstein_numbers_upper_bound_max_noncompact_eq3}, thus achieving the desired result.
\end{proof}

\begin{rem}
It follows from the equality \eqref{thm:bernstein_numbers_value} that all injective (strict) $s$-numbers of the embedding operator $E\colon V^1 L^{p,q}(\Sigma, \mu) \to L^{p^*,q}(\Sigma, \mu)$ coincide with the norm of the embedding. The reason is that Bernstein numbers are the smallest injective strict $s$-numbers (\citep[Theorem~4.6]{P:74}), that is,
\begin{equation*}
b_m(T) \leq s_m(T) \quad \text{for every $m\in\Ne$},
\end{equation*}
for every injective (strict) $s$-number $s$ and for every $T\in B(X,Y)$. A (strict) $s$-number is injective if the values of $s_n(T)$ do not depend on the codomain of $T$. More precisely, $s_n(J_N^Y \circ T) = s_n(T)$ for every closed subspace $N\subseteq Y$ and every $T\in B(X, N)$, where $J_N^Y\colon N \to Y$ is the canonical embedding operator.

Furthermore, the equality \eqref{thm:bernstein_numbers_value} also shows that all entropy numbers $e_m(E)$ of the embedding are equal to $\|E\|$. For $m\in\Ne$, the $m$th entropy number $e_m(E)$ is defined as
\begin{equation*}
e_m(E) = \inf \Big\{ \varepsilon > 0: B_{V^1 L^{p, q}(\Sigma, \mu)} \subseteq \bigcup_{j = 1}^{2^{m - 1}} \big(g_j + rB_{L^{p^*, q}(\Sigma, \mu)}\big), g_1,\dots, g_{2^{m - 1}} \in L^{p^*, q}(\Sigma, \mu) \Big\}.
\end{equation*}
It is easy to see that $\|E\| \geq e_1(E) \geq e_2(E) \geq \cdots \geq 0$ and $\lim_{m\to\infty}e_m(E) = \beta(E)$, which together with $\beta(E) = \|E\|$ implies that $e_m(E) = \|E\|$ for every $m\in\Ne$.
\end{rem}

\setcounter{tocdepth}{1}
\subsection*{Acknowledgment}
We would like to thank the referee for carefully reading the paper and their valuable comments.

\end{document}